\documentclass[]{interact}
\usepackage{enumerate}
\usepackage{epstopdf}
\usepackage[caption=false]{subfig}

\usepackage[numbers,sort&compress]{natbib}
\bibpunct[, ]{[}{]}{,}{n}{,}{,}
\makeatletter
\def\NAT@def@citea{\def\@citea{\NAT@separator}}
\makeatother
\usepackage{pstricks,graphicx}
\usepackage{float}
\usepackage{tikz}
\theoremstyle{plain}
\newtheorem{theorem}{Theorem}[section]
\newtheorem{lemma}{Lemma}[section]
\newtheorem{corollary}{Corollary}[section]

\newtheorem{definition}{Definition}[section]
\newtheorem{example}{Example}[section]

\theoremstyle{remark}
\newtheorem{remark}{Remark}

\newcommand{\argmin}{\operatornamewithlimits{argmin}}

\DeclareMathOperator{\con}{cone}

\DeclareMathOperator{\cl}{cl}
\DeclareMathOperator{\dom}{dom}
\DeclareMathOperator{\Int}{int}
\begin{document}

\title{Proper efficiency, scalarization and transformation in multi-objective optimization: Unified approaches}

\author{
\name{Moslem Zamani\textsuperscript{a,}\textsuperscript{b}\thanks{Email: zamani.moslem@tdt.edu.vn (Moslem Zamani)}and
 Majid Soleimani-damaneh\textsuperscript{c}\thanks{Email: soleimani@khayam.ut.ac.ir (Majid Soleimani-damaneh; Corresponding author)}}
\affil{\textsuperscript{a} Parametric MultiObjective Optimization Research Group, Ton Duc Thang University, Ho Chi Minh City, Vietnam;\\
\textsuperscript{b} Faculty of Mathematics and Statistics, Ton Duc Thang University, Ho Chi Minh City, Vietnam;\\
\textsuperscript{c} School of Mathematics, Statistics and Computer Science, College of Science, University of Tehran, Tehran, Iran; }
}

\maketitle

\begin{abstract}
In this paper, we investigate the relationships between proper efficiency and the solutions of a general scalarization problem in multi-objective optimization. We provide some conditions under which the solutions of the dealt with scalar program are properly efficient and vice versa. 
We also show that, under some conditions, if the considered general scalar problem is unbounded, then the original multi-objective problem does not have any properly efficient solution. In another part of the work, we investigate a general transformation of the objective functions which preserves proper efficiency. We show that several important results existing in the literature are direct consequences of the results of the present paper.

\end{abstract}


\section{Introduction}
Consider a general multi-objective optimization problem,
\begin{align}\label{MO}
\nonumber & \min \ f(x)\\
& \ s.t. \ x\in X,
\end{align}
where $f: X \to \mathbb{R}^p$ with $p\geq 2$ is the objective function. Multi-objective optimization problems  arise naturally in many applications in engineering, management, economics, finance, etc. Indeed, each decision making or optimization problem with more than two criteria or objectives can be cast as a multi-objective optimization problem.

Efficient solutions of Problem \eqref{MO} are defined as members of $X$ for which it is impossible to improve some objective(s) without deteriorating (at least) another one \cite{Eh, Mit}. Mathematically, $\bar x \in X$ is an efficient solution of Problem \eqref{MO} if there is no $x\in X$ with $f_i(x)\leq f_i(\bar x)$, $i=1, ..., p$ and $f_j(x)< f_j(\bar x)$ for some $j$. Proper efficiency is an important solution concept in multi-objective optimization which has been proposed in order to eliminate efficient solutions with unbounded trade-offs \cite{Geo, Mit}.

Scalarization is one of the most common approaches to handle multi-objective optimization problems \cite{Eh, Mit, Kass}. By scalarization methods, one solves a single-objective optimization problem, corresponding to \eqref{MO}, whose optimal solutions can be (weakly, properly) efficient for \eqref{MO}.
In addition, scalarization techniques are employed as a subproblem in iterative methods which generate an approximation of the efficient set \cite{Mar}, and also in interactive approaches which try to produce the most preferred solution \cite{Mit}.

An important question concerning scalarization problems is about the connection between their solutions and proper efficiency, as well as their ability to generate properly efficient solutions. Scalarization methods are not only strong tools to generate (properly) efficient solutions, but also provide valuable information about (the quality of) these solutions. In this study, we consider a general (unified) scalarization program, and provide some conditions under which the optimal solutions of the dealt with scalarization problem are properly efficient. We list some well-known scalarization techniques which satisfy the given sufficient conditions. Furthermore, we focus on parametric scalarization tools, and give sufficient conditions under which a parametric scalarization method is able to generate all properly efficient solutions.  We investigate the unbounded case separately, and establish that under some conditions the unboundedness of the considered general scalarization problem implies the emptiness of the set of properly efficient solutions.

Another part of the current study is devoted to investigation of a general transformation which maps objective functions preserving proper efficiency. Transformation of objective functions have been mainly proposed for normalization of objectives with different units  \cite{marler}. In addition, it has been exploited  to facilitate handling multi-objective problem, for instance by convexifying \cite{Li, Rome}. In this paper, we give sufficient conditions under which the set of properly efficient solutions of the original and transformed problems are the same. We show that several important results existing in the literature are direct consequences of the results of the present paper.

The rest of the paper is organized as follows. We review terminologies and notations in Section \ref{Sp}. Section \ref{Ss} is devoted to the scalarization methods. A unified transformation for multi-objective problems is studied in Section \ref{S.tr}. Section \ref{conc} contains a short conclusion.

\section{Terminologies and notations}\label{Sp}

The $p$-dimensional Euclidean space is denoted by $\mathbb{R}^p$. Vectors are considered to be column vectors and the superscript $T$ denotes the transpose operation. We denote the $i$-th component of a given vector $y$ by $y_i$. We use $e$ and $e^i$ to denote vector of ones and $i$-th unit coordinate vector, respectively.
The nonnegative orthant is denoted by $\mathbb{R}_{+}^p$.  For a set $Y\subseteq\mathbb{R}^p$, we use the notations $\Int(Y)$ and
$\con(Y)$ for the interior and the conic hull of  $Y$, respectively.

The notations $\leqq$, $\leq$ and  $<$ stand for the following orders on  $\mathbb{R}^p$ with $p\geq 2$,
\begin{align*}
& x\leqq y\Longleftrightarrow y-x\in\mathbb{R}^p_{+}, \\
& x \leq y\Longleftrightarrow y-x\in\mathbb{R}^p_{+}\setminus\{0\}, \\
& x<y\Longleftrightarrow y-x\in\Int(\mathbb{R}^p_{+}).
\end{align*}

Given a lower semi-continuous function $g: \mathbb{R}^p \to  \mathbb{R}\cup\{+\infty\}$ and $\bar y\in \dom g:=\{y : g(y)<+\infty\}$,  the regular subdifferential of $g$ at $\bar y$ is defined as
$$
 \hat\partial g(\bar y)=\{ \nu: \liminf_{\substack {y\to\bar y \\ y\neq \bar y}} \frac{g(y)-g(\bar y)- \langle \nu, y-\bar y\rangle}{\|y-\bar y\|}\geq 0\}.
$$
We remark that the regular subdifferential of a lower semi-continuous function $g$ at a given point $\bar y\in\dom g$ is a closed convex set. Furthermore, $\hat\partial g(\bar y)=\nabla g(\bar y)$ provided that $g$ is continuously differentiable at $\bar y$. We refer the reader to \cite{Mord} for a comprehensive study of  regular subdifferentials.

According to Rademacher's theorem, every locally Lipschitz function on $\mathbb{R}^n$ is almost everywhere differentiable in the sense of  Lebesgue measure \cite{Clar}. Let $\phi: \mathbb{R}^p \to\mathbb{R}^q$ be a locally Lipschitz function. The generalized Jacobian of $\phi$ at $\bar x$, denoted by $\partial \phi(\bar x)$, is defined by
$$
\partial \phi(\bar x):=co\{\lim_{\nu\to +\infty} \nabla \phi(x_\nu): x_\nu\to \bar x,~ x_\nu\notin  X_f\},
$$
where $X_f$ is the set of points at which $\phi$ is not differentiable, and $\nabla \phi(x_\nu)$ is the $q\times p$ Jacobian matrix of $\phi$ at $x_\nu$. If $\phi$ is continuously differentiable at $\bar y$, then $\partial \phi(\bar y)=\nabla g(\bar y)$.  See \cite{Clar} for more information on the generalized Jacobian.

The point $y^I\in\mathbb{R}^p$ in which $y_i^I = \min_{x\in X}f_i(x)$, $i = 1,2,...,p$, is called the ideal point of \eqref{MO}, and the point $y^U\in\mathbb{R}^p $ with $y^U<y^I$ is said a utopia point.

Several concepts for proper efficiency have been introduced in the literature. In what follows, we list some definitions which will be used in the sequel. For a comprehensive study of proper efficiency, the reader is referred to \cite{Gue}.

\begin{definition}\cite{Geo} \label{Geo}
A feasible solution $\bar x \in X$ is called a properly efficient solution of \eqref{MO} in the Geoffrion's sense, if it is efficient and there exist a real number $M > 0$ such that for all
$i  \in\{1,2,...,p\}$ and $x\in X$ with $f_i(x) < f_i(\bar x)$, there exists an index $j \in\{1,2,...,p\}$ with $f_j(x) > f_j(\bar x)$ and
$$
\frac{f_i (\bar x) - f_i(x) }{ f_j(x)- f_j(\bar x)}\leq M.
$$
\end{definition}

\begin{definition}\cite{BenP} \label{BenP}
A feasible solution $\bar x \in X$ is called a properly efficient solution of \eqref{MO} in the Benson's sense, if
$$
\cl\Big(\con\big(f(X)+\mathbb{R}^p_+-f(\bar x)\big)\Big)\cap (-\mathbb{R}^p_+)=\{0\}.
$$
\end{definition}

\begin{definition}\cite{Henig} \label{Henig}
A feasible solution $\bar x \in X$ is called a properly efficient solution of \eqref{MO} in the Henig's sense if  there exits a convex pointed cone $C$ with
 $R^p_+\setminus\{0\} \subseteq \Int(C)$ and
 $$
f(X) \cap   \big(f(\bar x) - C\big)= \{f(\bar x)\}.
 $$
\end{definition}

Definitions \ref{Geo}-\ref{Henig} for \eqref{MO} are equivalent \cite{Gue, Henig}. We remark that having different definitions for the proper efficiency turns out to be of value. In fact, a result can be easily derived from a given definition, while the proof of the same result with other definitions may be less obvious.

For $\delta\in\mathbb{R}_+$, set
$$
C_\delta=\{y\in\mathbb{R}^p: \langle e^i+\delta e, y\rangle\geqq 0, i=1, ..., p\}.
$$

\begin{remark}\label{p09}
For a convex cone $C$ with $R^p_+\setminus\{0\} \subseteq \Int(C)$, there exists some $\delta>0$ such that $R^p_+\subseteq C_\delta\subseteq C$; See \cite{Zam4}.
\end{remark}

\section{Scalarization and proper efficiency: A general umbrella}\label{Ss}

As mentioned earlier, a most common approach for tackling multi-objective optimization problems is scalarization. A general scalarization problem, associated with \eqref{MO}, is formulated as
\begin{align}\label{Sc}
\nonumber & \min \ g(f(x))\\
& \ s.t. \ x\in X,
\end{align}
where $g:Y\to \mathbb{R}\cup\{+\infty\}$ is a given function, in which $Y\subseteq \mathbb{R}^p$ satisfies $f(X)\subseteq Y$.


\begin{definition}\label{D1}
Let $Y\subseteq  \mathbb{R}^p$ be closed. We call a function $g:Y\to \mathbb{R}$ subdifferential-positive on $Y$ if it is lower semi-continuous on $Y$ and there exists some $\epsilon \in int(\mathbb{R}^p_+)$ such that
$$
y\in Y,\,\,\xi\in\hat\partial g(y)\Longrightarrow\xi\geqq \epsilon.
$$
In the above definition, we set $g(y)=+\infty$ for $y\notin Y$. The vector-valued function $\phi:Y\to \mathbb{R}^q$ is called subdifferential-positive on $Y$ if $\phi_i,~i=1,2,\ldots,q,$ is subdifferential-positive on $Y$.
\end{definition}

Now, we are ready to present the first result of the paper.
\begin{theorem}\label{T1}
Let $Y$ be a closed convex set with $f(X)\subseteq Y$, and let $g: Y\to \mathbb{R}$ be subdifferential-positive on $Y$. Then each optimal solution of \eqref{Sc} is a properly efficient solution for \eqref{MO}.
\end{theorem}
\begin{proof}
We first extend the lower semi-continuous function $g$. It is seen that $\bar g: \mathbb{R}\to \mathbb{R}\cup\{+\infty\}$ defined by
$$\bar g(y)= \left\{\begin{array}{ll}
      g(y), & y\in Y \\
      +\infty, & y\notin Y
   \end{array}\right.
$$
is lower semi-continuous. The underlying reason for the extending $g$ is to use the mean value theorem. For convenience, let $g=\bar g$.
By indirect proof assume that $\bar x$ is an optimal solution of \eqref{Sc} while it is properly efficient. By Benson's proper efficiency definition, there are
$\{x_\nu\}_\nu\subseteq X$, $\{d_\nu\}_\nu\subseteq \mathbb{R}^p_+$ and $\{t_\nu\}_\nu\subseteq \mathbb{R}_+$ such that
$$
\lim_{\nu\to \infty} t_\nu(f(x_\nu)+d_\nu-f(\bar x))=-d,
$$
where $0\neq d\in\mathbb{R}^p_+$. So, without loss of generality, one may assume
\begin{align}\label{T1.1}
\lim_{\nu\to \infty} \frac{f(x_\nu)-f(\bar x)}{\|f(x_\nu)-f(\bar x)\|}= -\bar d,
\end{align}
for some $0\neq \bar d\in\mathbb{R}^p_+$.  By mean value Theorem 4.13 in \cite{Mord}, for each $\nu$, there are sequences $\{y^k_\nu\}_k\subseteq Y$ and $\{\xi^k_\nu\}_k$ with $y^k_\nu \to y_\nu$ as $k\to\infty$,
$\xi^k_\nu\in \hat\partial g(y^k_\nu)$, and
\begin{equation}\label{o9087}
\liminf_{k\to \infty} \langle \xi^k_\nu, f(x_\nu)-f(\bar x)\rangle\geq g(f(x_\nu))-g(f(\bar x)),
\end{equation}
where  $y_\nu\in [f(x_\nu), f(\bar x)]$; Here, $[y_1, y_2]$ stands for the line segment joining $y_1$ and $y_2$ in $\mathbb{R}^p$. By \eqref{o9087}, as $\bar x$ is an optimal solution of \eqref{Sc}, we have
\begin{align}\label{T1.2}
\liminf_{k\to \infty} \langle \xi^k_\nu, \frac{f(x_\nu)-f(\bar x)}{\|f(x_\nu)-f(\bar x)\|}\rangle\geq 0.
\end{align}
By \eqref{T1.1} and subdifferential-positive property of $g$ on $Y$, for $\nu$ sufficiently large, we should have
\begin{align*}
\liminf_{k\to \infty} \langle \xi^k_\nu, \frac{f(x_\nu)-f(\bar x)}{\|f(x_\nu)-f(\bar x)\|}\rangle< 0,
\end{align*}
which contradicts \eqref{T1.2} and completes the proof.
\end{proof}

Theorem \ref{T1} may not hold when  $\epsilon$ in Definition \ref{D1} is not strictly positive. The following example casts light on this point.

\begin{example}
Consider the multi-objective problem
\begin{align}\label{MO1}
\nonumber & \min \ \begin{bmatrix} -e^x\\  -e^{-x}
\end{bmatrix}\\
& \ s.t. \ x\in\mathbb{R}.
\end{align}
This problem does not admit any properly efficient solution. Let $g: -\mathbb{R}_+^2\to \mathbb{R}$ be given by $g(y)=-y_1y_2$. It is readily seen that $\nabla g(y)\geqq 0$ for each $y\in (-\mathbb{R}_+^2)$ and $f(\mathbb{R})\subseteq -\mathbb{R}_+^2$, but
$$
\argmin \{g(f(x)): x\in\mathbb{R}\}=\mathbb{R}.
$$
\end{example}

Corollary \ref{1+} below, addresses the result of Theorem \ref{T1} for differentiable case.

\begin{corollary}\label{1+}
Let $Y$ be a closed convex set with $f(X)\subseteq Y$, and $g: Y\to \mathbb{R}$ be continuously differentiable. Assume that there exists $\epsilon \in int(\mathbb{R}^p_+)$ such that
$\nabla g(y)\geqq \epsilon$ for each $y\in Y$. Then each optimal solution of \eqref{Sc} is a properly efficient solution for \eqref{MO}.
\end{corollary}

In the following, we show that several important existing results concerning scalarization and proper efficiency are directly derived from Theorem \ref{T1}. Indeed, the following results have been proved for each scalarization method in the literature separately; here we give a unified framework (we provide an umbrella for several important existing results) via Theorem \ref{T1}.

\begin{itemize}
\item{\textbf{Weighted Sum method} \cite{Zade}:} The scalar program of this method is formulated as
\begin{align}\label{S1}
\nonumber & \min \ \lambda^Tf(x)\\
& \ s.t. \ x\in X,
\end{align}
where $\lambda\in \Int(\mathbb{R}^p_+)$. Setting $g(y):=\lambda^Ty$, we have $\nabla g(y)=\lambda$ and so by Corollary \ref{1+} it is easily seen that the optimal solutions of \eqref{S1} are properly efficient for \eqref{MO}.

\item{\textbf{Compromise programming} \cite{Gear}:} The scalar program of this method is written as
\begin{align}\label{S2}
\nonumber  & \min \ (\sum_{i=1}^p \lambda_i(f_i(x)-y^U_i)^p)^\frac{1}{p}\\
& \ s.t. \ x\in X,
\end{align}
where $\lambda\in \Int(\mathbb{R}^p_+)$, $p>1,$ and $y^U$ is a utopia point. As $$g(y):=(\sum_{i=1}^p \lambda_i(y-y^U_i)^p)^\frac{1}{p}$$ is subdifferential-positive on $Y=\{y: y\geq y^I\}$, Corollary \ref{1+} implies that the optimal solutions of \eqref{S2} are properly efficient for \eqref{MO}. Recall that $y^I$ is the ideal point.

\item{\textbf{Conic scalarization method} \cite{Kas1}:} The scalar problem of this method can be written as
\begin{align}\label{S3}
\nonumber  & \min \ \sum_{i=1}^p \lambda_i(f_i(x)-y^r_i)+\alpha \sum_{i=1}^p|f_i(x)-y^r_i|\\
& \ s.t. \ x\in X,
\end{align}
where $\lambda\in \Int(\mathbb{R}^p_+)$, $y^r\in \mathbb{R}^p$, and $\alpha\in \mathbb{R}_+$. By setting $g(y):=\sum_{i=1}^p \lambda_i(y_i-y^r_i)+\alpha \sum_{i=1}^p |y_i-y^r_i|$, the function $g$ is convex but not necessarily differentiable on $ \mathbb{R}^p$. It is seen that $g$ is subdifferential-positive on $\mathbb{R}^p$ for $0\leq \alpha< \lambda_i, (i=1, ..., p)$. Therefore, by Theorem \ref{T1}, the optimal solutions of \eqref{S3} are properly efficient for \eqref{MO} when $0\leq \alpha< \lambda_i, (i=1, ..., p)$.

\item{\textbf{Modified weighted Tchebycheff method} \cite{Kal}:} This method is written as
\begin{align}\label{S4}
\nonumber  & \min \ \max_{i}\{ \lambda_i(f_i(x)-y^U_i)+\alpha e^T(f(x)-y^U)\}\\
& \ s.t. \ x\in X,
\end{align}
where $\lambda\in \Int(\mathbb{R}^p_+)$, $\alpha>0$, and $y^U$ is a utopia point. By setting $g(y):=\max_{i}\{ \lambda_i(y_i-y^U_i)+\alpha e^T(y-y^U)$, the function $g$ is convex but not necessarily differentiable on $ \mathbb{R}^p$. It is seen that $g$ is subdifferential-positive on $ \mathbb{R}^p$. So, Theorem \ref{T1} implies that each optimal solution of
\eqref{S4} is properly efficient for \eqref{MO}.
\end{itemize}

As a multi-objective optimization problem generally has numerous properly efficient solutions, one important question in this context is under which conditions a parametric scalarization technique is able to generate each properly efficient solution. A general parametric scalarization problem can be written as
\begin{align}\label{PSc}
\nonumber (P_u):~~& \min \ g(f(x),u)\\
& \ s.t. \ x\in X,
\end{align}
where $u$ is a parameter in $U$. The set $Y\subseteq \mathbb{R}^p$ with $f(X)\subseteq Y$ and the function $g: Y\times U\to \mathbb{R}\cup \{+\infty\}$ are given. Theorem \ref{T2} below, gives sufficient conditions under which the parametric problem \eqref{PSc} generates all properly efficient solutions of \eqref{MO}.

\begin{theorem}\label{T2}
Let $f(X)\subseteq Y$. If for each $\bar y\in Y$ and each $\delta>0$, there exists some $u\in U$  with
\begin{equation}\label{2+}
\{y\in Y: g(y,u)< g(\bar y,u)\}\subseteq (\bar y-C_\delta),
\end{equation}
then the parametric problem \eqref{PSc} generates all properly efficient solutions of \eqref{MO}, i.e., if $\hat x$ is a properly efficient solution for \eqref{MO}, then $\hat x$ solves \eqref{PSc} for some $u\in U.$
\end{theorem}
\begin{proof}
Assume $\hat x$ is a properly efficient solution of \eqref{MO}. Due to the Henig's proper efficiency definition, invoking Remark \ref{p09}, there is some $\delta>0$ such that
$f(X)\cap (f(\hat x)-C_{\delta})=\{f(\hat x)\}$. Now, by setting $\bar y:=f(\hat x)$, and applying the assumption of the theorem, there exists some $u\in U$ such that
$$
\hat x \in \argmin\{g(f(x); u): x\in X\}.
$$
This completes the proof.
\end{proof}

As an application of Theorem \ref{T2}, we establish that conic scalarization method \cite{Kas1} produces all properly efficient solutions. To this end, it is enough to show that this method fulfills all assumptions of Theorem \ref{T2}. Let $\bar y\in \mathbb{R}^p$ and $\delta >0$ be given. We show that for $\lambda=e$ and $\alpha\in (\frac{1}{2\delta+1},1)$, and the reference point $y^r=\bar y$  we have
$$
\left\{y\in \mathbb{R}^p: \sum_{i=1}^p (y_i-\bar y_i)+\alpha \sum_{i=1}^p |y_i-\bar y_i|<0\right\}\subseteq (\bar y-C_\delta),
$$
or  equivalently
$$\left\{y\in \mathbb{R}^p: \sum_{i=1}^p y_i+\alpha \sum_{i=1}^p |y_i|<0\right\}\subseteq -C_\delta.$$
In addition,
 $$\left\{y\in \mathbb{R}^p: \sum_{i=1}^p y_i+\alpha \sum_{i=1}^p |y_i|<0\right\}=\left\{y\in \mathbb{R}^p: \sum_{i=1}^p y_i+ \alpha \sum_{i=1}^p \beta_iy_i<0,~ \beta\in\{-1, 1\}^p\right\}.$$
As for $j$,
$$e^j+\delta e=\frac{1}{2}(2\delta+1)\big(\sum_{i=1}^p e^i+ \frac{1}{2\delta+1} (e^j-\sum_{\substack {i=1 \\ i\neq j}}^p e^i)\big)$$
the inclusion follows from the fact that
$\left\{y\in \mathbb{R}^p: \sum_{i=1}^p y_i+\alpha_2 \sum_{i=1}^p |y_i|<0\right\}\subseteq \left\{y\in \mathbb{R}^p: \sum_{i=1}^p y_i+\alpha_1 \sum_{i=1}^p |y_i|<0\right\}$ for $0<\alpha_1<\alpha_2$.

In the same line, one can show that the modified weighted Tchebycheff method satisfies the conditions of Theorem \ref{T2}, and so, by suitable choosing $\lambda$ and $\alpha$, this technique is able to generate all properly efficient solutions.

By Theorem \ref{T1}, one can obtain a properly efficient solution. However, scalarization methods can be exploited to recognize nonexistence of the properly efficient solution set. Consider the following scalarization problem:
\begin{align}\label{Scu}
\nonumber & \min \ g(f(x))\\
& \ s.t. \ x\in X,\\
\nonumber & \  \ \ \ \ \   f(x)\leqq \epsilon,
\end{align}
in which $Y\subseteq \mathbb{R}^p$ is a given set containing $f(X)$. Furthermore, $g: Y\to \mathbb{R}\cup \{\infty\}$ is a lower semi-continuous function and $\epsilon \in \mathbb{R}^p$.

\begin{theorem}\label{T4}
If Problem \eqref{Scu} is unbounded, then multi-objective optimization Problem \eqref{MO} does not have any properly efficient solution.
\end{theorem}
\begin{proof}
The proof is by contradiction.  Suppose that  $\bar x$ is a properly efficient solution of \eqref{MO}. Due to the Henig proper efficiency, there exists a convex pointed cone $C\subseteq \mathbb{R}^p$ with $\mathbb{R}^p_+\setminus\{0\}\subseteq \Int(C)$ and  $(f(X)-f(\bar x))\cap (-C)=\{0\}$. We show that the set $\{y\in f(X):~y\leqq \epsilon\}$ is bounded. If not, there exist a nonnegative sequence $\{t_\nu\}_\nu$ and a sequence $\{d_\nu\}_\nu\subseteq \mathbb{R}^p_+$ such that $\|d_\nu\|=1$ for each $\nu$; and $t_\nu\to+\infty$ and $\{\epsilon-t_\nu d_\nu\}_\nu\subseteq  f(X)$. Without loss of generality,  we may assume $d_\nu\to \bar d\in \mathbb{R}^p_+\setminus \{0\}$. As $\mathbb{R}^p_+\setminus\{0\}\subseteq \Int(C)$,
for $\nu$ sufficiently large, we have
$$
\frac{1}{t_\nu}(\epsilon-t_\nu d_\nu-f(\bar x)) \in  -C\setminus\{0\},
$$
which contradicts  $(f(X)-f(\bar x))\cap (-C)=\{0\}$. Therefore, the set $\{y\in f(X):~y\leqq \epsilon\}$ is bounded. So, minimum of the lower semi-continuous function $g$ on
$cl(\{y: y\in f(X), y\leqq \epsilon\})$ is finite. This implies the finiteness of the optimal value of Problem \eqref{Scu} and completes the proof.
\end{proof}

As an application, we apply Theorem \ref{T4} to Benson's problem \cite{BenS} written as
 \begin{align}\label{B.S}
\nonumber \min \ & \sum_{i=1}^p f_i(x)\\
   s.t. \  & f(x)\leqq f(\bar x), \\
\nonumber   &  \ x\in X,
\end{align}
where $\bar x\in X$. Benson \cite{BenS} showed that, under convexity, if Problem \eqref{B.S} is unbounded, then Problem \eqref{MO} does not have any properly efficient solution. Soleimani-damaneh and Zamani  \cite{Zam1} established this result without convexity. In addition, Zamani \cite{Zam3} proved it for a general ordering cone. It is readily seen that the above-mentioned results reported in \cite{BenS,Zam1} follow from Theorem \ref{T4}.

Another scalarization technique, to which one can apply Theorem \ref{T4}, is Pascoletti-Serafini scalarization. It is known that a variety of important scalarization techniques can be modelled as special cases of Pascoletti-Serafini scalarization \cite{Eich, Kass}. This method is formulated as
 \begin{align}\label{D.S}
\nonumber \min \ & t\\
   s.t. \  & f(x)\leqq a+tr, \\
\nonumber   &  \ x\in X,
\end{align}
where $a\in\mathbb{R}^p$ and $r\in \mathbb{R}^p_{+}$. If \eqref{D.S} is unbounded, then
 \begin{align}
\nonumber \min \ & g(f(x))\\
   s.t. \  & f(x)\leqq \epsilon, \\
\nonumber   &  \ x\in X,
\end{align}
with $g(y)=\min\{ t: y\leqq a+tr\}$ is unbounded for some $\epsilon\in \mathbb{R}^p.$ The considered $g$ is lower semi-continuous. Hence, by Theorem \ref{T4}, one can infer that if Pascoletti-Serafini scalarization with $a\in \mathbb{R}^p$ and $r\in \mathbb{R}^p_+$ is unbounded, then Problem \eqref{MO} does not have any properly efficient solution.
\section{Proper efficiency and transformation}\label{S.tr}

In this section, we investigate the relationship between the multi-objective problem \eqref{MO} and its objective-transformed correspondence in regard to the proper efficiency. Let $Y\subseteq \mathbb{R}^p$ satisfying $f(X)\subseteq Y$ and vector-valued function $\phi: Y\to \mathbb{R}^q$ be given. An objective-transformed version of \eqref{MO}, invoking $\phi$, can be written as
\begin{align}\label{MO1}
\nonumber & \min \ \phi(f(x))\\
& \ s.t. \ x\in X.
\end{align}
In the sequel, we provide some sufficient conditions under which properly efficient solution sets of Problems \eqref{MO} and \eqref{MO1} coincide. It is known when $\phi$ is $\mathbb{R}^p_+$-transformation on $Y$, then the efficient solutions of  Problems \eqref{MO} and \eqref{MO1} are the same \cite{Hir}; A function $\phi: Y\to Z\subseteq\mathbb{R}^p$ is called $\mathbb{R}^p_+$-transformation on $Y$ if it is bijective and
\begin{align*}
&\bar y\leqq \hat y \Leftrightarrow \phi(\bar y)\leqq \phi(\hat y),~~\forall \bar y, \hat y\in Y
\end{align*}

By mean value theorem \cite{Mord}, if  bijective $\phi$ is subdifferential-positive on $Y$, then it is $\mathbb{R}^p_+$-transformation on $Y$.

By the following example, we show that $\mathbb{R}^p_+$-transformation property of $\phi$  on $Y$ is not sufficient for coincidence of the properly efficient solutions of  Problems \eqref{MO} and  \eqref{MO1}.
\begin{example}
Consider the multi-objective problem
\begin{align*}
\nonumber & \min \ \begin{bmatrix} x^2\\  x
\end{bmatrix}\\
& \ s.t. \ x\leq0.
\end{align*}
Let $\phi: \mathbb{R}_+\times (-\mathbb{R}_+)\to \mathbb{R}^2$ be given by $\phi(y)=(\sqrt{y_1},~y_2)^T$. It can be seen that $\phi$ is $\mathbb{R}^2_+$-transformation on $\mathbb{R}_+\times (-\mathbb{R}_+)$. The point $\bar x=0$ is not properly efficient, because there does not exist $\lambda \in \Int(\mathbb{R}^2_+)$ such that
$\bar x\in\argmin_{x\leq 0}\{\lambda_1x^2+\lambda_2x\}$; Notice that the considered multi-objective problem is convex. Nevertheless, $\bar x$ is properly efficient for the problem transformed by $\phi$. This follows form the fact that each efficient solution of a linear multi-objective optimization problem is properly efficient \cite{Eh}.
\end{example}

Throughout Lemma \ref{L1} and Theorem \ref{T3}, we provide sufficient conditions for equality of the properly efficient sets of Problems \eqref{MO} and \eqref{MO1}.

\begin{lemma}\label{L1}
Let $f(X)\subseteq Y$ be closed and convex and let $\phi: Y\to \mathbb{R}^q$  be subdifferential-positive on $Y$. Then the set of properly efficient solutions of \eqref{MO1} is a subset of that of \eqref{MO}.
\end{lemma}
\begin{proof}
Similar to the proof of Theorem \ref{T1}, we extend $\phi$ as follows,
$$\phi(y)=\left\{\begin{array}{ll}
      \phi(y), & y\in Y \\
      +\infty, & y\notin Y
   \end{array}\right.
$$
For convenience, let $\phi=\bar\phi$. By indirect proof, assume that $\bar x$ is a properly efficient solutions of \eqref{MO1}, while it is not a properly efficient solution of \eqref{MO}. By Benson's proper efficiency, there exist
$\{x_\nu\}_\nu\subseteq X$, $\{d_\nu\}_\nu\subseteq \mathbb{R}^p_+$ and $\{t_\nu\}_\nu\subseteq \mathbb{R}_+$ such that
$$
\lim_{\nu\to \infty} t_\nu(f(x_\nu)+d_\nu-f(\bar x))=-d,
$$
for some $d\in\mathbb{R}^p_+\setminus\{0\}$. Without loss of generality, one may assume
\begin{align}\label{T2.1}
\lim_{\nu\to \infty} \frac{f(x_\nu)-f(\bar x)}{\|f(x_\nu)-f(\bar x)\|}=-\bar d,
\end{align}
where $0\neq \bar d\in\mathbb{R}^p_+$.  For $i=1,...,q$, there are  sequences $\{y^k_{\nu, i}\}_k\subseteq Y$ and $\{ \xi^k_{\nu, i}\}_k$ such that $y^k_{\nu, i} \to y_{\nu,i}$ as $k\to\infty$,
$ \xi^k_{\nu, i}\in \hat\partial g(y^k_{\nu, i})$ and
\begin{align}\label{T2.2}
\liminf_{k\to \infty} \left\langle \xi^k_{\nu, i}, f(x_\nu)-f(\bar x)\right\rangle\geq \phi_i(f(x_\nu))-\phi_i(f(\bar x)),
\end{align}
where  $y_{\nu, i}\in [f(x_\nu), f(\bar x)]$. Consider the sequence $$s_\nu:=\frac{\phi(f(x_\nu))-\phi(f(\bar x))}{\|f(x_\nu)-f(\bar x)\|}.$$ This sequence  either has a cluster point or its norm tends to infinity. We investigate both cases separately, and show that both cases would lead to a contradiction. First, without loss of generality, suppose that  $s_\nu$ converges to some $-\hat d$. By \eqref{T2.1},   \eqref{T2.2} and subdifferential-positive property of $\phi$ on $Y$, we  have $\hat d \in \Int(\mathbb{R}^p_+)$, which contradicts the proper efficiency of $\bar x$ for \eqref{MO1}. For latter case, without loss of generality, one may assume that ${s_\nu}/{\|s_\nu\|}$ tends to some nonzero vector  $-\hat d$. Similarly, we get $ \hat d \in\mathbb{R}^p_+$, which contradicts the proper efficiency of $\bar x$ for problem \eqref{MO1}, and the proof is complete.
\end{proof}

In the next theorem, we present some sufficient conditions for equality of the properly efficient sets of Problems \eqref{MO} and \eqref{MO1}.
\begin{theorem}\label{T3}
Assume the following conditions:
\begin{enumerate}[i)]
\item
$Y_1, Y_2\subseteq \mathbb{R}^p$ are closed and convex.
\item
$\phi: Y_1\to Y_2$ is bijective, and both $\phi$ and $\phi^{-1}$ are subdifferential-positive on $Y_1$ and $Y_2$, respectively.
\item
$f(X)\subseteq Y_1$ and $\phi\big(f(X)\big)\subseteq Y_2$.
\end{enumerate}
Then the properly efficient solutions of \eqref{MO} and \eqref{MO1} are the same.
\end{theorem}
\begin{proof}
It follows from Lemma \ref{L1}.
\end{proof}

In \cite{Za1}, Zarepisheh et al. have proved, given integer $l>0$, the set of properly efficient solutions of \eqref{MO} coincides with that of the following problem
\begin{align*}
\nonumber & \min \
 \begin{bmatrix}
f_1(x)^l\\
 \vdots \\
 f_p(x)^l
\end{bmatrix}\\
& \ s.t. \ x\in X,
\end{align*}
provided that  $y^I>0$. This result follows from Theorem \ref{T3}. It is enough to consider $Y_1=\{y: y\geqq y^I\}$ and  $Y_2=\{y: y\geqq \bar y^{I}\}$, where
 $\bar y^{I}_i=(y^{I}_i)^l$ for $i=1, ..., p$, and $\phi(f(x))=(f_1(x)^l,\ldots,f_p(x)^l)^T.$

As we considered a general transformation, the conditions of Theorem \ref{T3} might be restrictive in some cases. In the following results, we give some milder conditions  for a locally Lipschitz  transformation. In this result, $\partial\phi(\cdot)$ stands for the generalized Jacobian of $\phi$. Furthermore, for matrix $M$, the inequality $M\geqq 0$ is componentwise.

\begin{lemma}\label{L2}
Assume the following conditions:
\begin{enumerate}[(i)]
\item
The properly efficient set of \eqref{MO} is non-empty.
\item
$Y$ is open and convex satisfying $\cl(f(X))\subseteq Y$.
\item
$\phi: Y\to \mathbb{R}^q$  is locally Lipschitz on $Y$.
\item
$\ker(M)\cap \mathbb{R}^p_+=\{0\}, \  \forall y\in Y, \forall M\in\partial{\phi}(y)$.
\item
$M\geqq 0, \  \forall y\in Y, \forall M\in\partial{\phi}(y)$.
\end{enumerate}
Then the set of properly efficient solutions of \eqref{MO1} is a subset of that of \eqref{MO}.
\end{lemma}
\begin{proof}
 The lemma is proved similar to Lemma \ref{L1}. Let  $\bar x$ be a properly efficient solution of \eqref{MO1}. To the contrary, assume that $\bar x$ is not a properly efficient solution of \eqref{MO}. By Benson's proper efficiency, there exist
$\{x_\nu\}_\nu\subseteq X$, $\{d_\nu\}_\nu\subseteq \mathbb{R}^p_+$ and $\{t_\nu\}_\nu\subseteq \mathbb{R}_+$ such that
$$
\lim_{\nu\to \infty} t_\nu(f(x_\nu)+d_\nu-f(\bar x))=-d,
$$
for some $d\in\mathbb{R}^p_+\setminus\{0\}$. Without loss of generality, one may assume
\begin{align}\label{T22.1}
\lim_{\nu\to \infty} \frac{f(x_\nu)-f(\bar x)}{\|f(x_\nu)-f(\bar x)\|}=-\bar d,
\end{align}
where $0\neq \bar d\in\mathbb{R}^p_+$.  By Theorem 8 in \cite{Hirr}, for each $\nu,$ there are  $\{y^1_\nu, ..., y^q_\nu\}\subseteq Y$ and $\lambda^\nu\in\mathbb{R}^q_+$ such that
\begin{align}\label{T33.2}
\phi(f(x_\nu))-\phi(f(\bar x))=\sum_{k=1}^q \lambda^\nu_kM^k_\nu(f(x_\nu)-f(\bar x)),
\end{align}
where  $y^k_\nu\in [f(x_\nu), f(\bar x)]$ and $M^k_\nu\in\partial \phi(y^k_\nu)$, $k=1, ..., q$, and $\sum_{k=1}^q \lambda_k^\nu=1$.  Suppose that $\{f(x_\nu)\}_\nu$ has a cluster point. Without loss of generality, one may assume that $\lambda^\nu\to \lambda$, $f(x_\nu)\to \bar y \in Y$, $y^k_\nu\to y_k  \in Y$ and $M^k_\nu\to M^k\in\partial \phi(y_\nu)$ for $k=1, ..., q$.  By the assumptions of the theorem, accompanying \eqref{T22.1} and \eqref{T33.2}, we get
\begin{align*}
\lim_{\nu\to\infty}\frac{\phi(f(x_\nu))-\phi(f(\bar x))}{\|f(x_\nu)-f(\bar x)\|}=-\hat d,
\end{align*}
for some $\hat d\in\mathbb{R}^q_+\setminus\{0\}$. The preceding relation contradicts the proper efficiency of  $\bar x$ for \eqref{MO1}. Now we consider the case that $\{f(x_\nu)\}_\nu$ is unbounded. Let $\hat x$ be a properly efficient solution of \eqref{MO}. One can infer from \eqref{T22.1},
\begin{align*}
\lim_{\nu\to \infty} \frac{f(x_\nu)-f(\hat x)}{\|f(x_\nu)-f(\bar x)\|}=-\bar d.
\end{align*}
This contradicts the proper efficiency of $\hat x$ for \eqref{MO}, and the proof is complete.
\end{proof}

In general, Lemma \ref{L2} does not hold when Problem \eqref{MO} does not have any properly efficient solution. The following example clarifies this point.

\begin{example}
Consider the multi-objective problem
$$\begin{array}{ll}
\nonumber & \min \ \left[\begin{matrix} f_1(x)\\  f_2(x)
\end{matrix}\right]\\
& \ s.t. \ x\leq1,
\end{array}$$
with $f_1(x)=x$ and
$$
f_2(x)=\left\{\begin{array}{ll}
      -x, & -1\leq x\leq 1 \\
      1,  &  x\leq -1 \\
   \end{array}\right.
$$
Let $\phi: \mathbb{R}^2\to \mathbb{R}^2$ given by $\phi(y)=\begin{bmatrix} e^{y_1},& e^{y_2}\end{bmatrix}^T$.
The original and the transformed problems have the same efficient solutions. Figure 1 illustrates that  all efficient points of the transformed problem are properly efficient while the original problem does not have any properly efficient solution.
\begin{figure}[H]\label{fig1}
\center
\begin{tikzpicture}[domain=0:4]
  \draw[->] (-3.2,0) -- (2.2,0) node[right] {$f_1$};
  \draw[->] (0,-2) -- (0,2) node[above] {$f_2$};

\draw[blue,domain=-1:1,samples=100] plot (\x, -\x);
\draw[blue,domain=-3.2:-1,samples=100] plot (\x, 1);
\end{tikzpicture}
\qquad
\begin{tikzpicture}[domain=0:4]
  \draw[->] (-1.2,0) -- (3.2,0) node[right] {$\phi_1(f)$};
  \draw[->] (0,-1) -- (0,3.2) node[above] {$\phi_2(f)$};
   \draw [fill=white] (0,2.718281828459045) circle[radius= 0.1 em];
\draw[blue,domain=-1:1,samples=100] plot ({exp(\x)}, {exp(-\x)});
\draw[blue,domain=-3.2:-1,samples=100] plot ({exp(\x)}, 2.718281828459045);
\end{tikzpicture}
\caption{ $f(X)$ and $\phi\big(f(X)\big)$}
\end{figure}
\end{example}

Note that in the same line one can establish Lemma \ref{L2} when $f(X)$ is Lipschitz arc-wise (arc-wise $\mathbb{R}^p_+$-convex) connected. A set $Y\subseteq \mathbb{R}^p$ is called Lipschitz arc-wise (arc-wise $\mathbb{R}^p_+$-convex) connected if for each $y_1, y_2\in Y$, there exists a Lipschitz (convex) function  $\gamma: [0, 1]\to Y$ with $\gamma(0)=y_1$ and $\gamma(1)=y_2$. Since convex functions on compact subsets of Euclidean spaces are Lipschitz,  arc-wise $\mathbb{R}^p_+$-convex connectivity implies Lipschitz arc-wise connectivity \cite{Clar}.

Corollary \ref{111+} below, addresses the result of Theorem \ref{L2} for differentiable case.

\begin{corollary}\label{111+}
Assume that the properly efficient set of \eqref{MO} is non-empty, and $Y$ is open and convex satisfying $\cl(f(X))\subseteq Y$. Furthermore, assume that $\phi: Y\to \mathbb{R}^q$  is continuously differentiable on $Y$. If $\nabla\phi(y)\geqq 0$ and $\{d\in\mathbb{R}^p_+\backslash \{0\}: \nabla\phi(y)d=0\}=\emptyset$ for each $y\in Y$, then the set of properly efficient solutions of \eqref{MO1} is a subset of that of \eqref{MO}.
\end{corollary}

In the next theorem, we give other sufficient conditions under which the properly efficient sets of Problems \eqref{MO} and \eqref{MO1} are the same.

\begin{theorem}\label{T5}
Assume the following conditions:
\begin{enumerate}[(i)]
\item
The properly efficient set of \eqref{MO} is non-empty.
\item
$Y_1, Y_2\subseteq \mathbb{R}^p$ are open and convex.
\item
$\phi: Y_1\to Y_2$ and its inverse, $\phi^{-1}$, are locally Lipschitz.
\item
$\cl(f(X))\subseteq Y_1$ and $\cl(\phi(f(X)))\subseteq Y_2$.
\item
$
\ker(M)\cap \mathbb{R}^p_+=\{0\}, \  \forall y\in Y_1, \forall M\in\partial{\phi}(y);
$\\
$M\geqq 0, \  \forall y\in Y, \forall M\in\partial{\phi}(y)$.
\item
$\ker(M)\cap \mathbb{R}^p_+=\{0\}, \  \forall y\in Y_2, \forall M\in\partial{\phi^{-1}}(y)$; \\
$M\geqq 0, \  \forall y\in Y, \forall M\in\partial{\phi^{-1}}(y)$.
\end{enumerate}
Then the properly efficient solutions of \eqref{MO} and \eqref{MO1} are the same.
\end{theorem}
\begin{proof}
It follows from Lemma \ref{L1}.
\end{proof}

A corollary similar to Corollary \ref{111+} can be written for Theorem \ref{T5} as well.

Hirschberger (Theorem 5.2 in \cite{Hir}) showed that both Problems \eqref{MO} and \eqref{MO1} share the same properly efficient solutions provided that the following conditions hold:
\begin{enumerate}[(a)]
\item
$f(X), \phi(f(X))\subseteq \mathbb{R}^p$ are closed and arc-wise $\mathbb{R}^p_+$-convex;
\item
$Y_1$ and $ Y_2$ are open sets with $ f(X)\subseteq Y_1$ and $ \phi(f(X))\subseteq Y_2$;
\item
$\phi: Y_1\to Y_2$ is a diffeomorphism (both $\phi$ and $\phi^{-1}$ are bijective and differentiable);
\item
$\phi: Y_1\to Y_2$ is $\mathbb{R}^p_+$-transformation;
\item
The properly efficient set of \eqref{MO} is non-empty.
\end{enumerate}
Since $\phi$ is $\mathbb{R}^p_+$-transformation, for given $\bar y\in Y$ and $d\in \mathbb{R}^p_+\setminus\{0\}$,
$$
\nabla\phi(\bar y)d=\lim_{t\to 0}\frac{\phi(\bar y+td)-\phi(\bar y)}{t}\geqq 0.
$$
As $d\in \mathbb{R}^p_+\setminus\{0\}$ is an arbitrary point, we must have $\nabla\phi(\bar y)\geqq 0$. In addition, $\phi$ is a diffeomorphism, thus $\nabla\phi(\bar y)$ is invertible and $\ker(\nabla\phi(\bar y))\cap \mathbb{R}^p_+=\{0\}$. Similarly, under these circumstances, we can also derive condition (vi) of Theorem \ref{T5}. Since the efficient set of \eqref{MO} is non-empty, $\mathbb{R}^p_+$-transformation properly implies that  the efficient set of \eqref{MO1} is also non-empty. Consequently, by Proposition 4.1 in  \cite{Hir}, the properly efficient set of \eqref{MO1} will be non-empty. As mentioned earlier, Theorem \ref{T5} holds under arc-wise $\mathbb{R}^p_+$-convex connectivity as well. So, Hirschberger's result follows from Theorem \ref{T5} when $\phi$ and $\phi^{-1}$ are locally Lipschitz on their domain.

In another paper, Zarepisheh and Pardalos  \cite{Za2} investigated some special classes of transformations. They considered the transformed problem

\begin{align}\label{Zar2}
\nonumber & \min \
 \begin{bmatrix}
g_1(f_1(x))\\
 \vdots \\
g_p(f_p(x))
\end{bmatrix}\\
& \ s.t. \ x\in X,
\end{align}
in which $g_i: [\inf_{x\in X}f_i(x), \sup_{x\in X}f_i(x)]\to\mathbb{R}$, $i=1, ..., p$.  This transformation is a special case of transformation $\phi$ investigated in Theorem \ref{T5}. They established  if $y^I\in \mathbb{R}^p$ exists and the following conditions are satisfied for each
$i=1, ..., p$, then the properly efficient solutions of Problems \eqref{MO} and \eqref{Zar2} are the same (Theorem 2 in \cite{Za2}):
\begin{enumerate}[(I)]
\item
$g_i$ is continuous on $[\inf_{x\in X}f_i(x), \sup_{x\in X}f_i(x)]$;
\item
$g_i$ is differentiable and $g_i^{\prime}$ is positive on $I_i:=\big(\inf_{x\in X}f_i(x), \sup_{x\in X}f_i(x)\big)$;
\item
Both $g_i$ and $g_i^{\prime}$ are increasing on $\big(\inf_{x\in X}f_i(x), \sup_{x\in X}f_i(x)\big)$.
\end{enumerate}
This result is correct when the interval considered in (II) and (III) is replaced with closed interval $\bar I_i:=\big[\inf_{x\in X}f_i(x), \sup_{x\in X}f_i(x)\big]$.
Indeed, assumptions (II) and (III) should be considered on a set containing $\bar I_i$. Then Theorem 2 in \cite{Za2} is a consequence of Theorem \ref{T5} of the current paper. The following example shows that the properly efficient solutions of \eqref{MO} and \eqref{Zar2} may not be the same if one assumes (I)-(III) with $I_i$'s instead of $\bar I_i$'s.

\begin{example}
Consider the multi-objective problem
\begin{align}\label{1111}
\nonumber & \min \ \begin{bmatrix} x\\  1-x
\end{bmatrix}\\
& \ s.t. \ 0\leqq x\leqq 1,
\end{align}
As the above problem is linear, all efficient solutions  are properly efficient \cite{Eh}. In addition, $\inf_{0\leqq x\leqq 1}f_1(x)=\inf_{0\leqq x\leqq 1}f_2(x)=0$. Let $\phi: \mathbb{R}_+^2\to \mathbb{R}^2$ be given by $\phi(y)=\begin{bmatrix} {y_1}^2,& {y_2}^4\end{bmatrix}^T$. It is easily seen that the example fulfills all assumptions (I)-(III) listed above. Here, $\bar x=1$ is a properly efficient solution of \eqref{1111}, but not for the transformed problem. This follows form the fact that the transformed problem is convex and
 $
 \bar x \notin \{\argmin \lambda_1x^2+\lambda_2(1-x)^4: 0\leqq x\leqq 1\}
 $
 for each $\lambda\in \Int(\mathbb{R}^2_+)$.
\end{example}

\begin{remark}
Some results of the paper are valid without lower semi-continuity assumption, though we considered this assumption throughout the paper for unification.
\end{remark}

\section{Conclusion}\label{conc}

In this paper, we provided some theorems for analysing a unified scalarization approach as well as a general objective transformation, regarding proper efficiency. In addition to establishing fundamental important results, we showed that several well-known results existing in the literature can be obtained as a by-product of these new theorems. These results not only provide a unified framework for examination of the scalarization techniques, but they pave the road for introducing and analysing new scalarization methods.

\bibliographystyle{tfs}
\bibliography{interactnlmsample}
\end{document}